\documentclass[11pt]{article}
\usepackage[margin=1in]{geometry} 
\usepackage{amssymb,amsfonts,amsmath,bbm,mathrsfs,stmaryrd}
\usepackage{mathabx}
\usepackage{xcolor}
\usepackage{url}

\usepackage{enumerate}

\usepackage{hyperref}
\hypersetup{colorlinks,
             linkcolor=black!75!red,
             citecolor=blue,
             pdftitle={},
             pdfproducer={pdfLaTeX},
             pdfpagemode=None,
             bookmarksopen=true
             bookmarksnumbered=true}

\usepackage{tikz}
\usetikzlibrary{arrows,calc,decorations.pathreplacing,decorations.markings,intersections,shapes.geometric,through,fit,shapes.symbols,positioning,decorations.pathmorphing}

\usepackage{braket}

\usepackage[amsmath,thmmarks,hyperref]{ntheorem}
\usepackage{cleveref}

\creflabelformat{enumi}{#2(#1)#3}

\crefname{section}{Section}{Sections}
\crefformat{section}{#2Section~#1#3} 
\Crefformat{section}{#2Section~#1#3} 

\crefname{subsection}{\S}{\S\S}
\AtBeginDocument{%
  \crefformat{subsection}{#2\S#1#3}%
  \Crefformat{subsection}{#2\S#1#3}% 
}

\crefname{subsubsection}{\S}{\S\S}
\AtBeginDocument{%
  \crefformat{subsubsection}{#2\S#1#3}%
  \Crefformat{subsubsection}{#2\S#1#3}% 
}

\theoremstyle{plain}

\newtheorem{lemma}{Lemma}[section]
\newtheorem{proposition}[lemma]{Proposition}

\newtheorem{theorem}[lemma]{Theorem}

\theoremstyle{nonumberplain}

\theoremstyle{plain}
\theorembodyfont{\upshape}
\theoremsymbol{\ensuremath{\blacklozenge}}

\newtheorem{definition}[lemma]{Definition}
\newtheorem{example}[lemma]{Example}

\crefname{definition}{definition}{definitions}
\crefformat{definition}{#2definition~#1#3} 
\Crefformat{definition}{#2Definition~#1#3} 

\crefname{ex}{example}{examples}
\crefformat{example}{#2example~#1#3} 
\Crefformat{example}{#2Example~#1#3} 

\crefname{remark}{remark}{remarks}
\crefformat{remark}{#2remark~#1#3} 
\Crefformat{remark}{#2Remark~#1#3} 

\crefname{convention}{convention}{conventions}
\crefformat{convention}{#2convention~#1#3} 
\Crefformat{convention}{#2Convention~#1#3} 

\crefname{notation}{notation}{notations}
\crefformat{notation}{#2notation~#1#3} 
\Crefformat{notation}{#2Notation~#1#3} 

\crefname{table}{table}{tables}
\crefformat{table}{#2table~#1#3} 
\Crefformat{table}{#2Table~#1#3}

\crefname{lemma}{lemma}{lemmas}
\crefformat{lemma}{#2lemma~#1#3} 
\Crefformat{lemma}{#2Lemma~#1#3} 

\crefname{proposition}{proposition}{propositions}
\crefformat{proposition}{#2proposition~#1#3} 
\Crefformat{proposition}{#2Proposition~#1#3} 

\crefname{corollary}{corollary}{corollaries}
\crefformat{corollary}{#2corollary~#1#3} 
\Crefformat{corollary}{#2Corollary~#1#3} 

\crefname{theorem}{theorem}{theorems}
\crefformat{theorem}{#2theorem~#1#3} 
\Crefformat{theorem}{#2Theorem~#1#3} 

\crefname{enumi}{}{}
\crefformat{enumi}{(#2#1#3)}
\Crefformat{enumi}{(#2#1#3)}

\crefname{assumption}{assumption}{Assumptions}
\crefformat{assumption}{#2assumption~#1#3} 
\Crefformat{assumption}{#2Assumption~#1#3} 

\crefname{equation}{}{}
\crefformat{equation}{(#2#1#3)} 
\Crefformat{equation}{(#2#1#3)}

\numberwithin{equation}{section}

\theoremstyle{nonumberplain}
\theoremsymbol{\ensuremath{\blacksquare}}

\newtheorem{proof}{Proof}
\newcommand\pf[1]{\newtheorem{#1}{Proof of \Cref{#1}}}

\newcommand\bC{{\mathbb C}}

\newcommand\bR{{\mathbb R}}
\newcommand\bS{{\mathbb S}}
\newcommand\bT{{\mathbb T}}

\newcommand\bZ{{\mathbb Z}}

\newcommand\cG{{\mathcal G}}

\newcommand\cP{{\mathcal P}}

\newcommand\fa{{\mathfrak a}}

\newcommand\fg{{\mathfrak g}}
\newcommand\fh{{\mathfrak h}}

\newcommand\fk{{\mathfrak k}}

\newcommand\fm{{\mathfrak m}}

\newcommand\fp{{\mathfrak p}}

\newcommand\fr{{\mathfrak r}}
\newcommand\fs{{\mathfrak s}}

\newcommand{\define}[1]{{\em #1}}

\newcommand{\qedhere}{\mbox{}\hfill\ensuremath{\blacksquare}}

\title{On the dearth of coproducts in the category of locally compact groups}
\author{Alexandru Chirvasitu}

\begin{document}

\date{}

\newcommand{\Addresses}{{% additional braces for segregating \footnotesize
  \bigskip
  \footnotesize

  \textsc{Department of Mathematics, University at Buffalo, Buffalo,
    NY 14260-2900, USA}\par\nopagebreak \textit{E-mail address}:
  \texttt{achirvas@buffalo.edu}

% %   \medskip
% %   
% %   \textsc{Department of Mathematics, institution,
% %     address}\par\nopagebreak \textit{E-mail address}:
% %   \texttt{??}
% % 
}}

\maketitle

\begin{abstract}
  We prove that a family of at least two non-trivial, almost-connected locally compact groups cannot have a coproduct in the category of locally compact groups if at least one of the groups is connected; this confirms the intuition that coproducts in said category are rather hard to come by, save for the usual ones in the category of discrete groups.

  Along the way we also prove a number of auxiliary results on characteristic indices of locally compact or Lie groups as defined by Iwasawa: that characteristic indices can only decrease when passing to semisimple closed Lie subgroups, and also along dense-image morphisms.
\end{abstract}

\noindent {\em Key words: locally compact group; Lie group; almost-connected; pro-Lie; coproduct; representation}

\vspace{.5cm}

\noindent{MSC 2020: 22D05; 22E46; 22D12; 18A30}

%\tableofcontents

%%%%%%%%%%%%%%%%%%%%%%%%%%%%%%%%%%%%%%%%%%%%%%%%%%%%%%%%%%%%%%%%%%%%%%%%%%%%%
%%%%%%%%%%%%%%%%%%%%%%%%%%%%%%%%%%%%%%%%%%%%%%%%%%%%%%%%%%%%%%%%%%%%%%%%%%%%%
\section*{Introduction}

Many of the familiar categories arising ``in nature'' happen to have {\it coproducts} \cite[\S III.3]{mcl} for families $\{G_i\}_{i\in I}$ of objects: an object $G$ that is the universal recipient of morphisms $G_i\to G$. This is true, for instance, of sets (where the coproduct is just the disjoint union), groups \cite[\S 4.6]{bg-inv}, rings \cite[\S 4.13]{bg-inv}, and in fact more generally, any {\it variety of algebras} in the sense of \cite[\S V.6]{mcl}; that is, the category of sets equipped with maps/relations satisfying of fixed ``shapes'' satisfying a fixed collection of equations (Lie algebras, abelian groups, monoids, etc. etc.). The latter result is a consequence of \cite[Theorem 9.3.8]{bg-inv}, for instance, or \cite[Remark 3.4 (4)]{ar}.

To illustrate the additional complications that obtain in handling {\it topological} algebraic structure, consider how one constructs the coproduct of a family $G_i$, $i\in I$ in the category of \define{compact} (Hausdorff) groups:
\begin{itemize}
\item form the group-theoretic coproduct $G$ of the $G_i$;
\item equip $G$ with the finest group topology making all $G_i\to G$ continuous;
\item then take the \define{Bohr compactification} of $G$ (for the latter construction, see e.g. \cite[\S 1, Corollary]{ah-bohr}).
\end{itemize}

The present paper is concerned with coproducts in the category $\mathcal{LCG}$ of {\it locally} compact (and always, for us, Hausdorff) topological groups. One case is easily dispatched: if the locally compact groups in question happen to all be discrete, simply form their usual coproduct in the category $\mathcal{GP}$ of (ordinary, non-topological) groups \Cref{le:discok}.

As soon as some non-discreteness seeps in things become more difficult. The main result is the no-go \Cref{th:aconn}:

\begin{theorem}\label{th:mainintro}
  A family of at least two almost-connected locally compact groups cannot have a coproduct in $\mathcal{LCG}$ if at least one of those groups is connected. \qedhere
\end{theorem}

As the statement indicates, we focus mainly on the well-behaved class of \define{almost-connected} \cite[Preface]{hm} locally compact groups $G$: those for which the quotient $G/G_0$ by the identity connected component is compact. These are in particular arbitrarily approximable by Lie groups, in the sense that every neighborhood of the identity contains a compact normal subgroup $N$ such that $G/N$ is Lie (e.g. \cite[\S 4.6, Theorem]{mz}); in \cite[Preface]{hm} this property is referred to as being \define{pro-Lie}.

The proof of \Cref{th:mainintro} relies on a convenient numerical invariant of a connected locally compact group $G$, introduced by Iwasawa \cite[Theorem 13]{iw}, which roughly speaking measures the group's departure from being compact: its {\it characteristic index}, denoted below by $\mathrm{ci}(G)$. This is the uniquely-determined non-negative integer $r$ for which $G$ admits a homeomorphic decomposition $G\cong K\times \bR^r$ for a maximal compact subgroup $K$ (see \Cref{def:ci}).

A number of estimates on $\mathrm{ci}$ are needed that seem to be difficult to extract from the literature in precisely the form needed here, so they are included here. An aggregate of \Cref{th:denseim} and \Cref{pr:cile}, for instance, reads

\begin{theorem}
  Let $f:H\to G$ be a morphism of connected locally compact groups.
  \begin{enumerate}[(a)]
  \item If $f$ has dense image then $\mathrm{ci}(H)\ge \mathrm{ci}(G)$.
  \item On the other hand, if $H$ is Lie semisimple then $\mathrm{ci}(H)\le \mathrm{ci}(G)$. \qedhere
  \end{enumerate}  
\end{theorem}

We also need an additivity result for characteristic indices under taking quotients by discrete subgroups. This appears as \Cref{pr:adddisc}:

\begin{proposition}
  For a connected locally compact group $G$ and a discrete central subgroup $Z\triangleleft G$ (which will automatically be finitely-generated) we have
  \begin{equation*}
    \mathrm{ci}(G) = \mathrm{rank}(Z)+\mathrm{ci}(G/Z). 
  \end{equation*}  
\end{proposition}
The statement supplements the analogue \Cref{eq:iwquot} for {\it connected} normal closed $N\trianglelefteq G$ \cite[Lemma 4.10]{iw}.

Topological coproducts are discussed extensively in the literature, but the discussion tends to branch in different directions than the one pursued here.
\begin{itemize}
\item On the one hand there are studies of \define{varieties} of topological groups, i.e. categories closed under a number of constructions, including coproducts (which are known to exist): e.g. \cite{gm-var} and references therein.
\item On the other, for locally compact groups the focus in the past seems to have been on whether or not one can equip (subgroups of) the standard coproduct with a locally compact topology: e.g. \cite[Corollary 1]{mn-fr}. 
\end{itemize}
In particular, \cite[Remark (1)]{mn-fr} is very similar in spirit to \Cref{th:aconn} applied to two circles, but argues that the standard coproduct $\bS^1*_{\mathcal{GP}}\bS^1$ in the category $\mathcal{GP}$ of groups has no ``non-obvious'' locally compact subgroups. This does not seem to say much about \Cref{th:aconn} because in principle $\bS^1*_{\mathcal{LCG}}\bS^1$ (if it existed) might be
\begin{itemize}
\item a completion of sorts
\item of some quotient of $\bS^1*_{\mathcal{GP}}\bS^1$
\end{itemize}
(as per the above description of {\it compact} coproducts).

%%%%%%%%%%%%%%%%%%%%%%%%%%%%%%%%%%%%%%%%%%%%%%%%%%%%%%%%%%%%%%%%%%%%%%%%%%%%%
\subsection*{Acknowledgements}

This work is partially supported by NSF grant DMS-2001128.

%%%%%%%%%%%%%%%%%%%%%%%%%%%%%%%%%%%%%%%%%%%%%%%%%%%%%%%%%%%%%%%%%%%%%%%%%%%%%
%%%%%%%%%%%%%%%%%%%%%%%%%%%%%%%%%%%%%%%%%%%%%%%%%%%%%%%%%%%%%%%%%%%%%%%%%%%%%
\section{Preliminaries}\label{se.prel}

All topological groups are assumed Hausdorff, and $\mathcal{LCG}$ is the category of locally compact groups.

It is a celebrated result of Iwasawa's that a connected locally compact group $G$ decomposes as
\begin{equation*}
  G=K H_1\cdots H_r\cong K\times H_1\times\cdots\times H_r
\end{equation*}
where
\begin{itemize}
\item $K$ is a maximal compact subgroup of $G$;
\item which is automatically connected and conjugate to any other maximal compact subgroup;
\item the $H_i$ are closed subgroups isomorphic to $(\bR,+)$;
\item and the homeomorphism `$\cong$' is given by the multiplication map
  \begin{equation*}
    K\times H_1\times\cdots\times H_r\to G
  \end{equation*}
  (though that map does \define{not} decompose $G$ as a direct product, as in general $K$ and the $H_i$ do not commute).
\end{itemize}
This is \cite[Theorem 13]{iw}, which contains an additional hypothesis on $G$ (that of being an ``(L)-group'', i.e. arbitrarily approximable by its Lie-group quotients) proven redundant in \cite[\S 4.6, Theorem]{mz}. Following Iwasawa:
\begin{definition}\label{def:ci}
  The \define{characteristic index} of a connected locally compact group $G$ is the non-negative integer $r$ in the discussion above.

  We will denote the characteristic index of $G$ by $\mathrm{ci}(G)$
\end{definition}

The fact that, as recalled in the Introduction, an almost-connected locally compact group $G$ has, for any neighborhood $1\in U\subseteq G$, a compact normal subgroup $K\subset U$ such that $G/K$ is Lie \cite[\S 4.6, Theorem]{mz} will be referenced repeatedly.

%%%%%%%%%%%%%%%%%%%%%%%%%%%%%%%%%%%%%%%%%%%%%%%%%%%%%%%%%%%%%%%%%%%%%%%%%%%%%
%%%%%%%%%%%%%%%%%%%%%%%%%%%%%%%%%%%%%%%%%%%%%%%%%%%%%%%%%%%%%%%%%%%%%%%%%%%%%
\section{Bounds on characteristic indices}

The present section collects a number of results, mostly revolving around characteristic-index computations and estimates, that have proven difficult to locate in the literature (at least in the form needed here, without further processing).

The characteristic index exhibits additivity under extensions
\begin{equation*}
  1\to N\to G\to G/N\to 1
\end{equation*}
of {\it connected} locally compact groups (with $N\trianglelefteq G$ closed): according to \cite[Lemma 4.10]{iw}, under such circumstances we have 
\begin{equation}\label{eq:iwquot}
  \mathrm{ci}(G) = \mathrm{ci}(N) + \mathrm{ci}(G/N). 
\end{equation}
We need some variations on this theme. For one thing, it is frequently convenient to pass from locally compact groups to their Lie quotients, or from arbitrary connected Lie groups to {\it linear} Lie groups (i.e. those admitting faithful finite-dimensional representations \cite[\S 1.5, discussion following Corollary 2]{ragh}, which is equivalent to being realizable as closed subgroups of $GL(n,\bR)$ for some $n$ \cite[Theorem 9]{goto2}).

Connected Lie groups often have linear quotients by discrete central subgroup: this happens when they are simply-connected \cite[\S 1.4]{ragh}, for instance, or semisimple (following from the simply-connected case together with the fact that quotients of semisimple linear Lie groups are again linear \cite[Lemma 9]{goto2}). For this reason, it would be convenient to have something like \Cref{eq:iwquot} for {\it discrete} (rather than connected) $N$.

Let us first begin with the following simple observation, well-known but set out here for future reference.

\begin{lemma}\label{le:isfg}
  A discrete normal subgroup of a locally compact connected group is automatically abelian finitely-generated.
\end{lemma}
\begin{proof}
  That such a subgroup $Z\trianglelefteq G$ is in fact {\it central} follows from its normality and discreteness, and the connectedness of $G$: the conjugation action of the latter will have connected orbits, but the only non-empty connected subsets of $Z$ are the singletons.

  As to finite generation, pass first to a Lie quotient $G/K$ of $G$ by a compact normal subgroup. This will affect nothing, since $K\cap Z$ must be finite (being both discrete and compact). But now $Z$ can be identified with the fundamental group of a compact manifold (a compact Lie group, in fact), hence the conclusion.
\end{proof}

\Cref{le:isfg} implies that any discrete normal $Z\trianglelefteq G$ is of the form 
\begin{equation}\label{eq:fgab}
  Z\cong (\text{the torsion subgroup of $Z$})\times \bZ^r,
\end{equation}
and hence we can refer, as usual (e.g. \cite[p.46]{lng}), to its {\it rank} $\mathrm{rank}(Z):=r$. With all of this in place, we can now proceed to the aforementioned discrete-subgroup analogue of \Cref{eq:iwquot}.

\begin{proposition}\label{pr:adddisc}
  Let $G$ be a connected locally compact group and $Z\triangleleft G$ a discrete normal subgroup. We then have
  \begin{equation}\label{eq:cici}
    \mathrm{ci}(G) = \mathrm{rank}(Z)+\mathrm{ci}(G/Z). 
  \end{equation}
\end{proposition}
\begin{proof}
  Since quotients by finite groups do not affect the characteristic index and we can work our way through the individual $\bZ$ factors in \Cref{eq:fgab} successively, we will assume $Z\cong \bZ$. The goal, then, is to prove that
  \begin{equation}\label{eq:plus1}
    \mathrm{ci}(G) = 1 + \mathrm{ci}(G/\bZ). 
  \end{equation}
  Consider a connected, closed, proper, normal subgroup $\{1\}\ne N\trianglelefteq G$. We can then apply the desired result to $N$ and $G/N$ separately: either $Z\cap N$ is infinite cyclic, in which case the proposition applies to {\it it}, or the image of $Z$ generates a closed subgroup of $G/N$ of the form
  \begin{equation}\label{eq:torz}
    (\bS^1)^m\times \bZ,\ m\ge 0
  \end{equation}
  so that
  \begin{equation*}
    \mathrm{ci}(G/N) = 1 + \mathrm{ci}((G/N)/\text{image of }Z).
  \end{equation*}
  Either way, having taken care of the smaller groups $N$ and $G/N$ in this manner, we derive the desired \Cref{eq:plus1} from \Cref{eq:iwquot}. This argument serves to reduce the problem to the case when $G$ has no such normal subgroups $N$, i.e. $G$ is $\bR$ (because it is 1-dimensional and has an embedded copy of $\bZ$) or simple.

  If $G\cong \bR$ we are done; this leaves the case of simple $G$, to which the rest of the proof is devoted.

  Begin by fixing a a global Cartan decomposition
  \begin{equation}\label{eq:cartdec}
    K\times \fp\ni (k,x)\stackrel{\cong}{\longmapsto} k\exp(x)\in G
  \end{equation}
  as in \cite[Theorem 6.31]{knp} with $K$ containing the center $Z(G)$ and hence also $Z$, and such that $K/Z(G)$ is maximal compact in $G/Z(G)$. Furthermore, the $\fp$ Cartesian factor in the decomposition \Cref{eq:cartdec} survives as a Cartesian factor upon quotienting by $Z$, so by subtracting $\dim\fp$ from both $\mathrm{ci}$ terms in \Cref{eq:cici} our goal becomes
  \begin{equation*}
    \mathrm{ci}(K) = 1+\mathrm{ci}(K/Z). 
  \end{equation*}
  The Lie algebra $\fk:=Lie(K)$ is compact in the sense of \cite[Definition 6.1]{hm} and hence of the form
  \begin{equation*}
    \fk\cong \bR^{m}\times \prod_{i=1}^n \fs_i
  \end{equation*}
  for simple Lie algebras $\fs_i$ \cite[Theorem 6.6]{hm}. Then, up to irrelevant quotienting by a finite central subgroup, we may assume $K$ is of the form
  \begin{equation*}
    K\cong \bR^{\mathrm{ci}(K)}\times \bT^{m'}\times\prod_{i=1}^n S_i
  \end{equation*}
  where
  \begin{itemize}
  \item the $S_i$ are simple compact Lie groups with respective Lie algebras $\fs_i$;
  \item $\mathrm{ci}(K)+m'=m$. 
  \end{itemize}
  In this setup it is clear that the quotient $K/Z$ will be of the form
  \begin{equation*}
    \bR^{\mathrm{ci}(K)-1}\times \bT^{m'+1}\times\text{(a compact semisimple group)}.
  \end{equation*}
  Since the characteristic index of this product is precisely $\mathrm{ci}(K)-1$, we are done.
\end{proof}

We next record a phenomenon (reminiscent of \Cref{eq:iwquot}) whereby $\mathrm{ci}$, if it changes at all, can only decrease along morphisms with dense image.

\begin{theorem}\label{th:denseim}
  For a dense-image morphism $f:G\to H$ of connected locally compact groups we have $\mathrm{ci}(G)\ge \mathrm{ci}(H)$.
\end{theorem}
\begin{proof}
  Consider a compact normal subgroup $K\trianglelefteq H$ with $H/K$ Lie, an open neighborhood $U$ of $1\in H/K$ that contains no non-trivial subgroups \cite[Chapter II, Exercise B.5]{helg}, and let $V$ be the pullback of $U$ through
  \begin{equation}\label{eq:ghhk}
    G\stackrel{f}{\longrightarrow}H\to H/K.
  \end{equation}
  $G$ contains a compact normal subgroup $K'\subset V$ with $G/K'$ Lie \cite[\S 4.0]{mz}, and the no-subgroup-in-$U$ condition ensures that \Cref{eq:ghhk} annihilates $K'$. Since furthermore quotienting out compact normal subgroups does not alter characteristic indices, we may as well substitute $G/K'$ and $H/K$ for $G$ and $H$ respectively; better yet, we can simply assume $G$ and $H$ are Lie to begin with.

  So long as we can find closed, connected, (non-trivial and proper) normal subgroups $N\trianglelefteq G$ we can
  \begin{itemize}
  \item apply the result recursively to $f|_N:N\to \overline{f(N)}$;
  \item and similarly to the morphism
    \begin{equation*}
      G/N\to H/\overline{f(N)};
    \end{equation*}
  \item then lifting back up to $f:G\to H$ via Iwasawa's additivity \Cref{eq:iwquot}. 
  \end{itemize}
  It thus suffices to consider the case when $G$ has no such normal subgroups $N$; when it is, in other words,
  \begin{enumerate}[(a)]
  \item\label{item:3} one-dimensional (and hence a circle or $\bR$)
  \item\label{item:4} or simple.
  \end{enumerate}
  For \Cref{item:3} observe that $H$ will also be abelian and connected and hence of the form
  \begin{equation*}
    (\bS^1)^m\times \bR^{\mathrm{ci}(H)}
  \end{equation*}
  for some $m$ \cite[Chapter II, Exercise C.2]{helg}. But then either $G\cong \bS^1$ and its image is automatically closed and hence $\mathrm{ci}(H)=0$, or $G\cong \bR$ and $\mathrm{ci}(H)\le 1$ because no morphism $\bR\to \bR^{\ge 2}$ can have dense image.

  This leaves case \Cref{item:4}: $G$ is simple. If there is an infinite cyclic discrete central subgroup
  \begin{equation*}
    \bZ\cong Z\le G
  \end{equation*}
  then the closed subgroup $\overline{f(Z)}\le H$ must be of the form \Cref{eq:torz} or a torus, so we can substitute
  \begin{equation*}
    G/Z\to H/\overline{f(Z)}
  \end{equation*}
  for $f$: the characteristic index of $G$ will have decreased by 1 by \Cref{pr:adddisc}, whereas that of $H$ will have decreased by {\it at most} 1.

  Repeating this procedure, we can exhaust the torsion-free component of the center of $G$ and thus assume that its center is finite. But semisimple Lie subgroups with finite center are automatically closed \cite[Theorem 2]{goto1}, so in that case $f(G)=H$ and we are back to using Iwasawa's original \Cref{eq:iwquot} with
  \begin{equation*}
    N:=\ker f\quad\text{and}\quad H=f(G)\cong G/N.
  \end{equation*}
  This concludes the proof.
\end{proof}

One initial attempt at a proof of \Cref{th:aconn} proceeded by producing, for two non-trivial simple linear Lie groups $G_i$, $i=1,2$, representations
\begin{equation*}
  \rho_i:G_i\to GL(V)
\end{equation*}
on a common space such that the group $G$ generated topologically by $\rho_i(G_i)$ contains unipotent subgroups of large dimension (meaning subgroups of $GL(V)$ consisting only of unipotent matrices, i.e. with spectrum $\{1\}$). One might hope that this then allows us to conclude that $G$ has large characteristic index, thus finishing the argument via \Cref{pr:lgind}.

More generally, intuition dictates that the characteristic index of a connected Lie group imposes an upper bound on the characteristic index of its Lie subgroups. We prove a version of this.

\begin{proposition}\label{pr:cile}
  For a closed embedding $H\le G$ of connected locally compact groups with $H$ Lie semisimple the characteristic index of $H$ is dominated by that of $G$.
\end{proposition}
\begin{proof}
  We make a number of simplifications. For one thing, $G$ too can be assumed Lie: simply pass to
  \begin{equation*}
    HK/K\to G/K
  \end{equation*}
  for a compact normal subgroup $K\trianglelefteq G$ with $G/K$ Lie, noting that that embedding is still closed because the kernel $K$ is compact.
  
  \vspace{.5cm}
  
  {\bf Reduction to the linear case.} We can always render $G$ linear by modding out a central closed subgroup $Z\le G$ (e.g. via the adjoint representation). Now, being Lie and abelian, $Z$ must be of the form
  \begin{equation}\label{eq:zrs}
    Z\cong \bZ^{m}\times(\text{finite abelian group}) \times \bR^n\times (\bS^1)^p
  \end{equation}
  for non-negative integers $m$, $n$ and $p$: the connected component decomposes as a Euclidean group times a torus by \cite[Chapter II, Exercise C.2]{helg}, and that component is divisible[\S 5, Definition]{kpl-inf} and hence splits off as a summand \cite[Theorem 2]{kpl-inf}.

  Annihilating normal compact subgroups makes no difference to characteristic indices, so we will henceforth ignore the compact component
  \begin{equation*}
    \text{finite}\times (\bS^1)^p
  \end{equation*}
  of \Cref{eq:zrs}. According to \Cref{eq:iwquot} and \Cref{pr:adddisc} we have
  \begin{equation*}
    \mathrm{ci}(G/Z) = \mathrm{ci}(G) - m - n.
  \end{equation*}
  On the other hand the image
  \begin{equation*}
    HZ/Z\cong H/H\cap Z
  \end{equation*}
  of $H$ through $G\to G/Z$ is now semisimple and linear, meaning that it is automatically closed \cite[Theorem 2, Lemma 5]{goto1}. In any decomposition
  \begin{equation*}
    Z\cap H\cong \bZ^{m'}\times \bR^{n'}
  \end{equation*}
  analogous to \Cref{eq:zrs} (again, ignoring compact factors) we must have $m'+n'\le m+n$, so by \Cref{eq:iwquot} and \Cref{pr:adddisc} again we will have
  \begin{equation*}
    \mathrm{ci}(H)-\mathrm{ci}(HZ/Z) = m'+n' \le m+n = \mathrm{ci}(G)-\mathrm{ci}(G/Z)
  \end{equation*}
  or, rearranging,
  \begin{equation*}
    \mathrm{ci}(G/Z) - \mathrm{ci}(H/Z) \le \mathrm{ci}(G) - \mathrm{ci}(H).
  \end{equation*}
  For that reason we have
  \begin{equation*}
    \mathrm{ci}(HZ/Z)\le \mathrm{ci}(G/Z)\Rightarrow \mathrm{ci}(H)\le \mathrm{ci}(G):
  \end{equation*}
  the linear-case inequality entails the general one.

  \vspace{.5cm}

  {\bf Reducing to semisimple $G$.} This is similar to the previous portion of the argument but simpler, now that we have specialized to linear groups. The additivity relation \Cref{eq:iwquot} applied to $N=R:=\mathrm{Rad}(G)$ (the radical) in particular gives 
  \begin{equation*}
    \mathrm{ci}(G/R) = \mathrm{ci}(G)-\mathrm{ci}(R)\le \mathrm{ci}(G).
  \end{equation*}
  The intersection $H\cap R$ is normal and solvable and hence discrete ($H$ being semisimple), so it must in fact be finite because $H$ is linear \cite[Lemma 5]{goto1}. This means that $HR/R\cong H/H\cap R$ has the same characteristic index as $H$, so it will be enough to prove
  \begin{equation*}
    \mathrm{ci}(HR/R)\le \mathrm{ci}(G/R).
  \end{equation*}
  Note that $HR/R\le G/R$ is closed again by \cite[Theorem 2, Lemma 5]{goto1}, since it is both semisimple and linear: the semisimplicity is one of the hypotheses, the linearity of $H$ is being assumed, and that of its quotient $HR/R\cong H/H\cap R$ follows from \cite[Lemma 9]{goto2}.

  The upshot is that we can assume $G$ to be semisimple.

  \vspace{.5cm}
  
  {\bf Linear, semisimple $G$.} For a global Iwasawa decomposition
  \begin{equation*}
    H\cong K_H\times A_H\times N_H,
  \end{equation*}
  the characteristic index $\mathrm{ci}(H)$ is $\dim(A_H) + \dim(N_H)$, so it will be enough to argue that
  \begin{enumerate}[(a)]
  \item\label{item:1} $\dim(N_H)\le \dim(N_G)$
  \item\label{item:2} and $\dim(A_H)\le \dim(A_G)$
  \end{enumerate}
  where similarly, $G\cong K_G\times A_G\times N_G$ is an Iwasawa decomposition. We address the two claims separately.

  \vspace{.5cm}
  
  {\bf \Cref{item:1}: $\dim(N_H)\le \dim(N_G)$.} To see this, note first that the elements of $N_H$ remain unipotent in $G$ when regarded as operators on the Lie algebra $\fg:=Lie(G)$ via the adjoint representation, because morphisms of semisimple Lie algebras preserve Jordan decompositions (e.g. \cite[\S 6.4, Corollary]{hum}).

  But this then means that $N_H$ is conjugate to a subgroup of $A_GN_G$ \cite[\S 2.8]{mst-max}, whose unipotent elements are precisely those of $N_G$. In short: $N_H$ is conjugate to a subgroup of $N_G$, implying the requisite dimension inequality.

  \vspace{.5cm}
  
  {\bf \Cref{item:2}: $\dim(A_H)\le \dim(A_G)$.} Equivalently, we will work with the Lie algebras $\fa_{H}:=Lie(A_H)$ and similarly for $G$. The argument is very similar in spirit to what we saw above: $\fa_G$ is maximal \define{real diagonal} in the sense of \cite[\S 1.3]{mst-max} (i.e. its elements are semisimple with real eigenvalues) so by loc.cit. it will be enough to argue that $\fa_H\subset \fg$ is again real diagonal (for it will then be conjugate in $G$ to a Lie subalgebra of $\fa_G$; see also \cite[Theorem 6.51 and surrounding discussion on Iwasawa-decomposition uniqueness]{knp}).

  That $\fa_H$ consists of semisimple operators on $\fg$ follows from the Jordan-decomposition preservation, as before. The issue is proving that those operators have \define{real} eigenvalues on $\fg$. To that end, note that
  \begin{itemize}
  \item since $\fg$ is a finite-dimensional representation of $\fh$, the eigenvalues in question are obtained by evaluating various integral weights \cite[\S 13.1]{hum} of the complexification $\fh^{\bC}$ on $\fa_H$;
  \item said weights are rational combinations of the roots of $\fh^{\bC}$ \cite[\S 13.1 and \S 13.2, Table 1]{hum};
  \item and the roots take real values on $\fa_H$ \cite[discussion preceding Proposition 6.60]{knp}.
  \end{itemize}
  This finishes the proof.
\end{proof}

%%%%%%%%%%%%%%%%%%%%%%%%%%%%%%%%%%%%%%%%%%%%%%%%%%%%%%%%%%%%%%%%%%%%%%%%%%%%%
\subsection{Side-note on linear quotients}

In the course of the proof of \Cref{pr:cile} we worked with a linear quotient of $G$ by some central subgroup. The latter, being abelian and closed in a connected Lie group, was of the form \Cref{eq:zrs}:
\begin{equation*}
  (\text{finitely-generated abelian group})\times (\text{Euclidean group})\times (\text{torus}).
\end{equation*}
As it turns out, this was overly cautious: one can always arrange to have the Euclidean factor absent. It seems worthwhile to record this here.

\begin{proposition}\label{pr:fgtor}
  A connected Lie group $G$ has a linear quotient by a central subgroup of the form
  \begin{equation}\label{eq:ator}
    A\times (\bS^1)^m,\ m\in \bZ_{\ge 0}
  \end{equation}
  for a finitely-generated discrete abelian group $A$.
\end{proposition}
\begin{proof}
  Let $R\le G$ be the radical, and $M\le G$ of $G$ a {\it Levi factor}: a Lie subgroup (not necessarily closed, in general) whose underlying Lie algebra
  \begin{equation*}
    \fm:=Lie(M)\le \fg:=Lie(G)
  \end{equation*}
  is a supplementary summand to $\fr:=Lie(R)$ (e.g. \cite[Chapter VI, Theorem 4.1]{ser-lie} and \cite[\S 1.3]{ragh}).

  Connected Lie groups are linear as soon as their radicals and Levi factors are \cite[Theorem 7]{goto2}, so we will focus on achieving that by quotienting out a central subgroup of the desired form.

  Consider first the central subgroup $Z\le M$ obtained as the common kernel of {\it all} finite-dimensional $M$-representations. It will in particular act trivially on $R$ by conjugation, so $Z$ and its closure $\overline{Z}\le G$ must be central.

  Although in general $Z$ will not be closed (i.e. $Z\ne \overline{Z}$), it is nevertheless the case that $\overline{Z}$ is of the form \Cref{eq:ator}: this follows, for instance, from \cite[Theorem 1]{goto1} and its proof. We thus have a first quotient $G/\overline{Z}$ by a closed central subgroup of the desired form \Cref{eq:ator} that at least has a linear Levi factor. All subsequent quotients we consider will be by images (in $G/\overline{Z}$) of tori contained in $R$ that were central in $G$ to begin with, so we may as well simplify the discussion by assuming that the Levi factor $M$ was linear to begin with.

  Given this latter assumption, we only need to render the radical $R$ linear. According to \cite[Theorem 5]{goto2} we would be done if the closed derived subgroup $C(R):=\overline{D(R)}$ were simply-connected. Since $C(R)$ is nilpotent (being the closure of the nilpotent Lie group $D(G)$ \cite[Chapter V, Corollary 5.3]{ser-lie}), the only it could possibly fail to be simply-connected would be for its (unique, central-in-$G$: \cite[\S 1.8]{ragh}, \cite[Lemma 13]{goto2}) maximal compact subgroup $K$ to be a non-trivial torus.

  We can now pass to $G/K$: the linearity of the Levi factor $M$ has not been affected, as quotients of linear semisimple Lie groups are linear \cite[Lemma 9]{goto2}, and since the kernel $K$ is compact, closures of quotients are quotients of closures, etc., so that $C(G/K) = C(G)/K$. The latter is now simply-connected, and we are done.
\end{proof}

In general, the toral factor cannot be left out in \Cref{pr:fgtor}. The following example given in \cite[Appendix]{goto1} of a non-closed Lie-group embedding will serve to illustrate this. 

\begin{example}\label{ex:nector}
  Let $\widetilde{SL(2,\bR)}$ be the universal cover of $SL(2,\bR)$. Its center will be isomorphic to $\bZ$, and the kernel of the surjection back to $SL(2,\bR)$ is generated by a central element $\sigma$. Now set
  \begin{equation*}
    G:=\widetilde{SL(2,\bR)}\times \bR/\{(\sigma^{n+m},n+m\gamma)\ |\ m,n\in \bZ\}
  \end{equation*}
  for some irrational number $\gamma\in \bR$.

  A finite-dimensional representation of $\widetilde{SL(2,\bR)}$ factors through $SL(2,\bR)$ (as follows, for instance, from the representation theory of the complexified Lie algebra $\mathfrak{sl}(2,\bC)$ \cite[\S 7.2]{hum}), so any morphism $G\to GL(n,\bR)$ must annihilate $\sigma\in SL(2,\bR)\subset G$. But the closure of $\langle \sigma\rangle$ in $G$ contains the circle
  \begin{equation*}
    \bR/\{n-n\gamma\ |\ n\in \bZ\}\subset G,
  \end{equation*}
  so that closure will be an embedded copy of $\bZ\times \bS^1$. The resulting morphism
  \begin{equation*}
    \widetilde{SL(2,\bR)}\subset G\to G/\bZ\times \bS^1
  \end{equation*}
  is nothing but the original quotient
  \begin{equation*}
    \widetilde{SL(2,\bR)}\to SL(2,\bR)
  \end{equation*}
  (or rather is isomorphic to it). 
\end{example}

The issue of whether or not a Lie group has a discrete-kernel linear quotient relates to the preceding discussion in other ways too. Several times (e.g. in the proofs of \Cref{th:denseim} and \Cref{pr:cile}) we have used the fact that semisimple Lie groups with finite center are automatically closed \cite[Theorem 2]{goto1}, which in turn entails the fact that semisimple Lie subgroups of linear groups are closed \cite[Corollary to Theorem 2]{goto1}. This can be enhanced slightly.

\begin{lemma}\label{le:disclin}
  Let $G$ be a connected Lie group admitting a linear quotient $G/Z$ by a discrete central subgroup. Any connected semisimple Lie subgroup of $G$ is automatically closed.
\end{lemma}
\begin{proof}
  This is immediate from the already-cited \cite[Corollary to Theorem 2]{goto1}: $H\le G$ is connected semisimple and the semisimple subgroup
  \begin{equation*}
    \mathrm{im}(H) = HZ/Z\cong H/H\cap Z\le G/Z
  \end{equation*}
  is closed in the linear group $G/Z$. But then so is its preimage $HZ$ through $G\to G/Z$, so the identity component $H=(HZ)_0$ must be closed too.
\end{proof}

This unifies a number of other results similar results. For instance, semisimple Lie subgroups are automatically closed when the large ambient group $G$ is
\begin{itemize}
\item simply-connected (\cite[Corollary 1]{mst-ext}), because simply-connected groups have discrete-kernel linear quotients \cite[\S 1.4]{ragh};
\item or compact (\cite[Corollary 2]{mst-ext}), since compact Lie groups are already linear \cite[Corollary 2.40]{hm}.
\end{itemize}

%%%%%%%%%%%%%%%%%%%%%%%%%%%%%%%%%%%%%%%%%%%%%%%%%%%%%%%%%%%%%%%%%%%%%%%%%%%%%
%%%%%%%%%%%%%%%%%%%%%%%%%%%%%%%%%%%%%%%%%%%%%%%%%%%%%%%%%%%%%%%%%%%%%%%%%%%%%
\section{Colimits of locally compact groups}

\begin{theorem}\label{th:aconn}
  A family of at least two non-trivial almost-connected locally compact groups with at least one connected cannot have a coproduct in $\mathcal{LCG}$. 
\end{theorem}

It will require some detours and surrounding scaffolding, on which we now embark. First, a converse of sorts:

\begin{lemma}\label{le:discok}
  If $\Gamma_i$, $i\in I$ is a family of discrete groups then their coproduct in the category of groups is also their coproduct in $\mathcal{LCG}$.
\end{lemma}
\begin{proof}
  This is clear: the embedding functor of the category $\mathcal{GP}$ of discrete groups into $\mathcal{LCG}$ is left adjoint to the discretization functor $\mathcal{LCG}\to \mathcal{GP}$, and hence it preserves colimits \cite[(dual of) \S V.5, Theorem 1]{mcl}.
\end{proof}

The following simple remark will, on occasion, allow us to consider simpler groups. 

\begin{lemma}\label{le:quots}
  If a family $\{G_i\}_{i\in I}$ has a coproduct in $\mathcal{LCG}$ then so does any family of quotients $G_i\to \overline{G_i}$.
\end{lemma}
\begin{proof}
  This is straightforward: if $G=\bigstar_{\mathcal{LCG}} G_i$ denotes the coproduct of the original family, its quotient by the smallest closed normal subgroup generated by all
  \begin{equation*}
    \ker\left(G_i\to \overline{G_i}\right)
  \end{equation*}
  obviously satisfies the universality property required of the coproduct $\bigstar_{\mathcal{LCG}} \overline{G_i}$.
\end{proof}

Next, we isolate the following notion that will recur in subsequent discussions.

\begin{definition}\label{def:unb}
  A family $\cG:=\{G_i\ |\ i\in I\}$ of topological groups is \define{jointly characteristic-index-unbounded} or \define{(jointly) $\mathrm{ci}$-unbounded} for short if there are morphisms
  \begin{equation}\label{eq:phis}
    \varphi_i:G_i\to L
  \end{equation}
  into locally compact groups wherein the connected component of the closed subgroup $G_{\varphi_i}$ generated by $\varphi_i(G_i)$ has arbitrarily large characteristic index.

  $\cG$ is {\it linearly} (jointly) $\mathrm{ci}$-unbounded if it satisfies the above condition with morphisms into {\it linear} groups $L$ (or equivalently, $L=GL(V)$ for finite-dimensional vector spaces $V$).

  If $\cG$ is linearly $\mathrm{ci}$-unbounded and families \Cref{eq:phis} exist with $G_{\varphi_i}$ containing connected semisimple Lie groups of arbitrarily large characteristic index, then $\cG$ is {\it semisimply} $\mathrm{ci}$-unbounded, or (jointly) {\it $\mathrm{ci}_{ss}$-unbounded}. 
\end{definition}

We will make use of the following device at least twice, so it is perhaps worth highlighting.

\begin{definition}\label{def:compr}
  Consider two morphisms $\varphi_i:G_i\to G$ into a locally compact group, and suppose we would rather have $\varphi$ land in an open subgroup of $G$ having a property $\cP$ (e.g. being pro-lie or almost-connected) which is closed under
  \begin{itemize}
  \item taking open subgroups,
  \item and finite-index extensions.
  \end{itemize}
  Then, assuming that
  \begin{itemize}
  \item one of $G_i$, $G_1$, say, is non-discrete;
  \item $G$ is similarly non-discrete;
  \item and $G$ does have a property-$\cP$ open subgroup $H\le G$,
  \end{itemize}
  we can first consider the open preimage
  \begin{equation*}
    H_1:=\varphi_1^{-1}(H)\le G_1,
  \end{equation*}
  which will be non-trivial because $G_1$ is assumed non-discrete. The analogous preimage in $G_2$ might be trivial if $G_2$ is finite, but in that is the case we can always replace the initial open subgroup $H\le G$ with
  \begin{equation*}
    \left(\bigcap_{g\in G_2} \varphi_2 g H\varphi_2 g^{-1}\right)\varphi_2(G_2) \le G,
  \end{equation*}
  which will still have property $\cP$ by assumption. This alteration, though, ensures that when $G_2$ is finite we can simply take $H_2:=G_2$ and still have $\varphi_2(H_2)\le H$.

  We will have thus found a property-$\cP$ open subgroup $H\le G$ and non-trivial open subgroups $H_i\le G_i$ for which
  \begin{equation*}
    \varphi(H_i)\le H. 
  \end{equation*}
  Under these conditions we will say that the original morphisms $\varphi_i$ have been \define{compressed} to
  \begin{equation*}
    \varphi_i|_{H_i}:H_i\to H
  \end{equation*}
  respectively, and what we have gained is the fact that the codomain now has the desired property $\cP$. Furthermore, if $G_2$ is finite we can simply take $H_2=G_2$.
\end{definition}

The difference between the two versions of \Cref{def:unb} vanishes in the cases of interest here.

\begin{lemma}\label{le:islinunb}
  Let $\cG:=\{G_i,\ i=1,2\}$ be a pair of almost-connected locally compact groups with $G_1$ infinite. If $\cG$ is $\mathrm{ci}$-unbounded in the sense of \Cref{def:unb} then it is linearly $\mathrm{ci}$-unbounded.
\end{lemma}
\begin{proof}
  We are assuming that $G_i$ jointly map to a locally compact group $L$ where they generate a closed subgroup with large characteristic index. Our assumption that $G_1$ is infinite allows us to use the compression procedure in \Cref{def:compr}: we can then assume that $L$ is connected an hence pro-Lie, at the cost of replacing $G_i$ with open non-trivial subgroups $H_i\le G_i$.

  By further taking a quotient we can assume that $L$ is linear an Lie. In short, we have representations $\rho_i:H_i\to GL(V)$ for a finite-dimensional $V$ such that $\rho_i(H_i)$ generate a group with large $\mathrm{ci}$.

  $H_i\le G_i$ are open non-trivial subgroups of almost-connected groups, so they must have finite index. We can thus induce $\rho_i$ to a representation
  \begin{equation*}
    \sigma_i:=\mathrm{ind}_{H_i}^{G_i}\rho_i
  \end{equation*}
  on a finite-dimensional space $V_i$ \cite[\S 2.1]{kt}. Furthermore, by the very definition of the induced representation, the restriction $\sigma_i|_H$ contains a copy of $\rho_i$ operating on an isomorphic copy $V\le V_i$.
 
  By replacing $V_i$ with $V_i^{\oplus n_i}$ for appropriately-chosen $n_i$ if necessary, we can assume that both $\sigma_i$ operate on the same space $W$. Then, upon conjugating $\sigma_2$ by an invertible operator on $W$, we can assume that we have isomorphic copies
  \begin{equation*}
    \sigma_i|_{H_i}\cong \rho_i,\ i=1,2
  \end{equation*}
  operating on the same subspace $V\le W$. All in all, we have
  \begin{itemize}
  \item morphisms $\sigma_i:G_i\to GL(W)$;
  \item whose restrictions to $H_i$ leave a subspace $V\le W$ invariant;
  \item such that the resulting subgroup
    \begin{equation*}
      \overline{\langle \sigma_i(H_i)\rangle}\subseteq GL(V)
    \end{equation*}
    has large characteristic index. 
  \end{itemize}
  But that index will then be dominated by that of 
  \begin{equation*}
    \overline{\langle \sigma_i(G_i)\rangle}\subseteq GL(W),
  \end{equation*}
  which must thus also be large. But we saw above that this is precisely what was needed to complete the proof.
\end{proof}

Note, in passing, that the induction-based argument used in \Cref{le:islinunb} also proves

\begin{lemma}\label{le:finindok}
  Let $\{G_i,\ i=1,2\}$ be a pair of locally compact groups and $H_i\le G_i$ two finite-index subgroups. If $\{H_i\}$ is linearly $\mathrm{ci}$-unbounded then that unboundedness can be witnessed by finite-dimensional representations
  \begin{equation*}
    G_i\to GL(V). 
  \end{equation*}
\end{lemma}

Returning to coproducts, the concept of joint $\mathrm{ci}$-unboundedness is germane for the following reason.

\begin{proposition}\label{pr:lgind}
  Let $G_i$, $i=1,2$ be two almost-connected locally compact groups such that
  \begin{itemize}
  \item $G_1$ is infinite;
  \item and such that all pairs of non-trivial open subgroups $H_i\le G_i$ are jointly $\mathrm{ci}_{ss}$-unbounded in the sense of \Cref{def:unb}.
  \end{itemize}
  Then, $G_i$ do not have a coproduct in the category $\mathcal{LCG}$.
\end{proposition}
\begin{proof}
  We will argue by contradiction, assuming $G_i$ \define{do} have a coproduct $G$ in $\mathcal{LCG}$.
  
  Like all locally compact groups, $G$ has an open subgroup $H\le G$ containing the identity component $G_0$ such that $H/G_0$ is compact and hence approximable by Lie groups \cite[\S 4.0]{mz}. By the compression procedure of \Cref{def:compr} (which requires the assumed infinitude of $H_1$) we then have non-trivial open subgroups $H_i\le G_i$ mapped to $H$ by the canonical morphisms $G_i\to G$.

  Now, our $\mathrm{ci}_{ss}$-unboundedness assumption allows us, via \Cref{le:finindok}, to construct finite-dimensional representations $\rho_i:G_i\to GL(V)$ such that $H_i$ topologically generate linear groups $S\le GL(V)$ containing connected semisimple subgroups with large characteristic index. The restrictions $\rho_i|_{H_i}$ factor through the almost-connected open subgroup $H\le G$, and hence the connected component
  \begin{equation*}
    \overline{\langle \varphi_i(H_i)\rangle}\le GL(V)
  \end{equation*}
  will be mapped to densely by $H_0=G_0$ (\Cref{le:connsurj}). In summary:
  \begin{equation*}
    \begin{tikzpicture}[auto,baseline=(current  bounding  box.center)]
      \path[anchor=base] 
      (0,0) node (u) {$H_0$}
      +(0,-2) node (d) {$S$}
      +(2,-1) node (r) {$\overline{\langle \varphi_i(H_i)\rangle}$}
      ;
      \draw[->] (u) to[bend left=6] node[pos=.5,auto] {$\scriptstyle \text{dense image}$} (r);
      \draw[->] (d) to[bend right=6] node[pos=.5,auto,swap] {$\scriptstyle \text{closed embedding}$} (r);
    \end{tikzpicture}
  \end{equation*}
  We \Cref{th:denseim} and \Cref{pr:cile} respectively prove
  \begin{equation*}
    \mathrm{ci}(H_0)\ge \mathrm{ci}(\overline{\langle \varphi_i(H_i)\rangle})\quad \text{and}\quad \mathrm{ci}(\overline{\langle \varphi_i(H_i)\rangle})\ge \mathrm{ci}(S),
  \end{equation*}
  which render $\mathrm{ci}(H_0)$ unbounded and deliver the desired contradiction.
\end{proof}

\begin{lemma}\label{le:connsurj}
  Let $G$ be an almost-connected locally compact group and $\varphi:G\to L$ a morphism of locally connected groups. The restriction
  \begin{equation*}
    \varphi|_{G_0}:G_0\to \left(\overline{\varphi G}\right)_0
  \end{equation*}
  to the identity component must then have dense image.
\end{lemma}
\begin{proof}
  Restricting our attention to $\overline{\varphi G}$, we will assume without loss of generality that $\varphi:G\to L$ has dense image to begin with. The goal is then to show that so does
  \begin{equation*}
    \varphi|_{G_0}:G_0\to L_0. 
  \end{equation*}
  Now, $\varphi G_0$ is normal in $\varphi G$, so it must be normal in its closure $L$ as well. But then so is its closure $\overline{\varphi G_0}$, and $\varphi$ induces a dense-image morphism
  \begin{equation}\label{eq:pfpf}
    G/G_0\to L/\overline{\varphi G_0}. 
  \end{equation}
  Because by assumption $G/G_0$ is profinite and quotients of profinite groups are profinite \cite[Exercise E1.13]{hm}, the codomain $L/\overline{\varphi G_0}$ of \Cref{eq:pfpf} will in particular be totally disconnected (in addition to being compact). But this was our goal to begin with: showing that $\overline{\varphi G_0}$ contains the connected component $L_0$.
\end{proof}

Next, we turn to the hypotheses of \Cref{pr:lgind} and whether or not they are satisfied in the cases of interest.

\begin{proposition}\label{pr:areciunb}
  Pairs $\{G_i,\ i=1,2\}$ of non-trivial locally compact groups with $G_1$ connected and $G_2$ almost-connected are jointly $\mathrm{ci}_{ss}$-unbounded in the sense of \Cref{def:unb}.
\end{proposition}
\begin{proof}
  The argument proceeds in a number of stages.

  {\bf Step 1: simplifying the connected case.} Since connected locally compact groups are pro-Lie \cite[\S 4.6, Theorem]{mz}, whenever we assume one of $G_i$ is connected we may as well take a quotient and assume it is Lie (and non-trivial, and connected). 
 
  Furthermore, the Lie algebra $\fg_i:=Lie(G_i)$ will then either be solvable or have a non-trivial semisimple quotient (\cite[Proposition 1.12 and subsequent discussion]{knp}). In the former case $\fg_i$ has $\bR$ as a quotient and hence $G_i$ surjects onto a circle group; in the latter $G_i$ surjects onto a \define{simple} connected Lie group (via the Lie-group-Lie-algebra correspondence \cite[Theorem 5.20]{hl-bk}, because its Lie algebra does \cite[Theorem 7.8]{hl-bk}). In summary, when dealing with connected $G_i$ we can (and will) assume that the group in question is either $\bS^1$ or connected, simple and linear.

  {\bf Step 2: $G_i$ both connected. } That is, circles or simple, by Step 1. We can then find
  \begin{itemize}
  \item circles $T_i\le G_i$;
  \item and finite-dimensional representations $\rho_i:G_i\to GL(V_i)$;
  \item in which, for some large $n$, the circles $T_i$ operate with at least $n$ eigenvalues (i.e. characters) $\chi_{i,j}$, $1\le j\le n$ that are sufficiently generic, i.e. in the sense that for each $i=1,2$ the characters
    \begin{equation*}
      \{\chi_{i,j}\}_j \quad\text{and}\quad \{\chi_{i,j}\chi_{i,j'}^{-1}\}_{j\ne j'}
    \end{equation*}
    are all distinct and non-trivial.
  \end{itemize}
  Taking multiple copies of each $V_i$ if necessary we can assume that $\rho_i$ are realized on a common space $V$, and then conjugating $\rho_2$ inside $GL(V)$, we can further assume that $T_i$ operate with respective eigenvalues $\chi_{i,j}$ on $n$ common lines
  \begin{equation*}
    \bC e_{j}\subset V,\ 1\le j\le n. 
  \end{equation*}
  By construction the $T_i$ restrict to {\it strongly regular} circle subgroups of $L:=PSL(\mathrm{span}\{e_j,\ 1\le j\le n\})$ in the sense of \cite[Definition 3.1]{2106.11955v1}, and by \cite[Theorem 3.5]{2106.11955v1} some pair of conjugates of those groups will generate $L$. Since the latter can be chosen to have arbitrarily large characteristic index, we are done. 

  {\bf Step 3: The connected-profinite case. } That is, one of the groups ($G_1$, say) is assumed connected and the other profinite and non-trivial. We can then substitute a {\it finite} non-trivial quotient for $G_2$, and then a finite cyclic subgroup thereof via \Cref{le:finindok}. In short, assume
  \begin{equation*}
    G\cong \bZ/n=\langle \sigma\rangle,\ n\ge 2.
  \end{equation*}
  But then
  \begin{equation}\label{rtimeszn}
    G_1 * G_2 \cong G_1 * (\bZ/n)\cong \left(\Asterisk_{i=0}^{n-1} Ad_{\sigma^i} G_1\right)\rtimes (\bZ/n),
  \end{equation}
  and we can fall back on the argument from Step 2: represent two of the conjugates $Ad_{\sigma^i} G_1$ appropriately, then induce that representation to \Cref{rtimeszn}.

  To conclude, note that upon applying Step 3 to $G_1$ and $G_2/G_{2,0}$ we conclude that the latter is trivial, i.e. $G_2$ is connected. But then Step 2 applies.
\end{proof}

\pf{th:aconn}
\begin{th:aconn}
  Immediate from \Cref{pr:lgind,pr:areciunb}.
\end{th:aconn}

%%%%%%%%%%%%%%%%%%%%%%%%%%%%%%%%%%%%%%%%%%%%%%%%%%%%%%%%%%%%%%%%%%%%%%%%%%%%%
%%%%%%%%%%%%%%%%%%%%%%%%%%%%%%%%%%%%%%%%%%%%%%%%%%%%%%%%%%%%%%%%%%%%%%%%%%%%%

%\bibliography{bib}{}
%\bibliographystyle{plain}
\def\polhk#1{\setbox0=\hbox{#1}{\ooalign{\hidewidth
  \lower1.5ex\hbox{`}\hidewidth\crcr\unhbox0}}}

\addcontentsline{toc}{section}{References}

\Addresses

\end{document}